\documentclass[12pt,oneside]{article}
\usepackage{amssymb}

\usepackage[latin1]{inputenc}
\usepackage[english]{babel}
\usepackage{amsmath}

\newtheorem{theorem}{Theorem}[section]

\newtheorem{corollary}[theorem]{corollary}

\newtheorem{proposition}[theorem]{Proposition}
\newtheorem{remark}{Remark}[section]

\newenvironment{proof}[1][Proof]{\textbf{#1.} }{\ \rule{0.5em}{0.5em}}

\begin{document}

\title{{\bf
 Singular Riesz measures on symmetric cones}\\
 {\small (Running title: \textbf{Singular Riesz measures})}}
\author {Abdelhamid Hassairi\footnote{Corresponding author.
 \textit{E-mail address: Abdelhamid.Hassairi@fss.rnu.tn}}$\;$  , and
 Sallouha
 Lajmi
\\{\footnotesize{\it
Sfax University, Tunisia.}}\\
    {\footnotesize{\it  }}}

 \date{}
 \maketitle
$\overline{\hspace{16.5cm}}$\vskip0.3cm

{\small {\bf Abstract}} { In this paper, we give an explicit
description of a class of positive measures on symmetric cones
defined by their Laplace transforms in the framework of the Riesz
integrals. This work is motivated by the importance of these
measures in probability theory and statistics since
they represent a generalization of the measures generating the famous Wishart distributions.}\\


$\overline{\hspace{15cm}}$\vskip1cm

\section{Introduction }
Many interesting results of analysis on Jordan algebras and their
symmetric cones have been not only used as powerful mathematical
tools but also as sources of inspiration in the development of other
fields and more particularly of probability theory and Statistics.
This seems to be due in particular to the importance, in certain
areas of these fields, of the special case of the algebra of
symmetric matrices and of its symmetric cone of positive definite
matrices. For instance in 2001, Hassairi and Lajmi ([3]) have
introduced a class of natural exponential families of probability
distributions generated by measures related to the so-called Riesz
integrals in analysis on symmetric cones (see [1], p.137). These
measures and the generated probability distributions have been
respectively  called by these authors Riesz measures and Riesz
probability distributions. The Riesz measures on a symmetric cone
are in fact defined by their Laplace transforms in a fondamental
theorem due to Gindikin ([2]). Roughly speaking, the theorem says
that the generalized power $\Delta_{s}(\theta^{-1})$ defined on a
symmetric cone is the Laplace transform of a positive measure
$R_{s}$ if and only if $s$ is in a given set $\Xi$ of $\Bbb{R}^{r}$,
where $r$ is the rank of the algebra. When $s$ is in a well defined
part of $\Xi$, the Riesz measure $R_{s}$ is absolutely continuous
with respect to the Lebesgue measure on the symmetric cone and has a
density which is expressed in terms of the generalized power. For
the other elements $s$ of $\Xi$, the Riesz measure $R_{s}$ is
concentrated on the boundary of the cone, we will say that they are
singular Riesz measures. These measures were considered of
complicated nature and their structure has never been explicitly
determined although some among them have a probabilistic
interpretation and play an important role in multivariate
statistics. The aim of the present paper is to give an explicit
description of the Riesz measure $R_{s}$ for all $s$ in $\Xi$. The
question is very interesting from a mathematical point of view, in
fact, besides the use of many important known facts from the
analysis on symmetric cones, we have been led to develop many other
useful results. On the other hand, we think that the knowledge of
the way in which a singular Riesz measure is built should allow us
to give a statistical interpretation of the generated family of
probability distributions extending the one corresponding to the
Wishart.

\section{Preliminaries}
In this section, we first recall some facts concerning Jordan
algebras and their symmetric cones, for more details, we refer the
reader to the book of  Faraut and Kor\'{a}nyi (1994),([1]) which is
a complete and self-contained exposition on the subject. We then
establish some new results on symmetric cones which we need in the
description of the Riesz measures.

Recall that a Euclidean Jordan algebra is a Euclidean space $V$ with scalar product $%
\langle x,y\rangle $ and a bilinear map
\[
V\times V\rightarrow V,\text{ }\quad (x,y)\longmapsto x.y
\]
called Jordan product such that, for all $x,y,z$ in $V,$

\qquad i) $x.y=y.x,$

\qquad ii) $\left\langle x,y.z\right\rangle =\left\langle
x.y,z\right\rangle ,$

\qquad iii) there exists $e$ in $V$ such that $e.x=x,$

\qquad iv)$\;x.(x^{2}.y)=x^{2}.(x.y),\qquad $where we used the abbreviation $%
x^{2}=x.x.$

A Euclidean Jordan algebra is said to be simple\textit{\ }if it does
not contain a nontrivial ideal. Actually to each Euclidean simple
Jordan algebra, one attaches the set of Jordan squares
\[
\overline{\Omega }=\left\{ x^2;x\in V\right\} .
\]
Its interior $\Omega $ is a symmetric cone i.e. a cone which is

\qquad i) self dual, i.e., $\Omega =\left\{ x\in V;\text{ \quad
}\langle x,y\rangle >0\text{ \quad }\forall y\in \overline{\Omega
}\setminus \left\{ 0\right\} \right\} $

\qquad ii) homogeneous, i.e. the subgroup $G(\Omega )$ of the linear group $%
GL(V)\;$of linear automorphisms which preserve $\Omega $ acts
transitively on $\Omega .$

\qquad iii) salient, i.e., $\Omega $ does not contain a line.
Furthermore, it is irreducible in the sense that it is not the
product of two cones.

Let now $x$ be in $V$. If $L(x)$ is the endomorphism of $V$ ;
$y\longmapsto x.y$ and $P(x)=2L(x)^2-L(x^2)$, then $L(x)$ and $P(x)$
are symmetric for the Euclidean structure of $V$, the map
$x\longmapsto $ $P(x)$ is called the quadratic representation of
$V.$

A element $c$ of $V$ is said to be idempotent if $c^2=c$, it is a
primitive idempotent if furthermore $c\neq 0$ and is not the sum
$t+u$ of two non null idempotents $t$ and $u$ such that $t.u=0.$

A Jordan frame is a set $\left\{ c_{1},c_{2},.....c_{r}\right\} $
such that $\displaystyle\sum_{i=1}^r c_i=e$ and $c_{i}.c_{j}=\delta
_{i,j}c_{i},$ for $1\leq i,j\leq r.$ It is an important result that
the size $r$ of such a frame is a constant called the rank of $V$.
For any element $x$ of a Euclidean simple Jordan algebra, there
exists a Jordan frame $(c_i)_{1\leq i \leq r}$ and
$(\lambda_1,...,\lambda_r)\in \Bbb{R}^{r}$ such that
$x=\displaystyle\sum_{i=1}^r\lambda_i c_i$. The real numbers
$\lambda_1,\lambda_2,...,\lambda_r$ depend only on $x$, they are
called the eigenvalues of $x$ and this decomposition is called its
spectral decomposition. The trace and the determinant of $x$ are
then respectively defined by
$\textrm{tr}(x)=\displaystyle\sum_{i=1}^r\lambda_i$ and $\det
x=\displaystyle\prod_{i=1}^r\lambda_i$. If $c$ is a primitive
idempotent of $V$, the only possible eigenvalues of $ L(c)$ are 0 ,
$\frac{1}{2}$ and 1. The corresponding spaces are respectively
denoted by $V(c,0),$ $V(c,\frac{1}{2})$ and $V(c,1)$ and the
decomposition
\[
V=V(c,1)\oplus V(c,\frac{1}{2})\oplus V(c,0)
\]
is called the Peirce decomposition of $V$ with respect to $c.$ An
element $x$ of $V$ can then be written in a unique way as
$$x=x_{1}+x_{12}+x_{0}$$
with $x_{1}$ in $V(c,1)$, $x_{12}$ in $V(c,\frac{1}{2})$ and $x_{0}$
in $V(c,0)$, which is also called the Peirce decomposition of $x$
with respect to the idempotent $c$. We will denote $\Omega_{c}$ the
symmetric cone associated to the sub-algebra $V(c,1)$ and $\det_{c}$
the
determinant in this sub-algebra.\\
Suppose now that  $(c_{i})_{1\leq i\leq
r}$ is  a Jordan frame in $V$ and, for 1$\leq i,j\leq r,$ we set

\[
V_{ij}^{{}}=\left\{
\begin{array}{c}
V(c_{i},1)=\Bbb{R}c_{i}\qquad \quad \text{if }i=j \\
V(c_{i},\frac{1}{2})\cap V(c_{j},\frac{1}{2})\quad \text{if }i\neq j
\end{array}
\right.
\]
Then (See [1], Th.IV.2.1) we have the Peirce decomposition
$V=\displaystyle\bigoplus_{i\leq j} V_{ij}$  with respect to the
Jordan frame ($c_{i})_{1\leq i\leq r}.$ The dimension of $V_{ij}$
is, for $i\neq j,$ a constant $d$ called the Jordan constant, it is
related
to the dimension $n$ and the rank $r$ of $V$ by the relation $n=r+\frac{d}{2}%
r(r-1)$.\\
For $1\leq k\leq r,$ we have
$$ V(c_1+...+c_k, 1)= \bigoplus_{i\leq j \leq k}V_{ij} \, ,\,\,  V(c_1+...+c_k, \frac{1}{2})= \bigoplus_{1\leq i \leq k < j}V_{ij}$$
In the following proposition, we establish some useful intermediary
results.
\begin{proposition}
Let $c$ be an idempotent of $V$. Then

i) $\Omega_{c}=P(c)(\Omega)$

ii) for all $x$ in $V(c,1)$, $2L(x)_{|V(c,\frac{1}{2})}$ is an
endomorphism of $V(c,\frac{1}{2})$ with determinant equal to
$\det_{c}(x)^{d(r-k)}$, where $k$ is the rank of $c$

iii) if $x$ in $V(c,1)$ is invertible, then
$2L(x)_{|V(c,\frac{1}{2})}$ is an automorphism of $V(c,\frac{1}{2})$
with inverse equal to $2L(x^{-1})_{|V(c,\frac{1}{2})}$.

iv) for all $x$ in $V(c,1)$, $L(x^2)_{|V(c,\frac{1}{2})}=
\frac{1}{2} L(x^{2})_{|V(c,\frac{1}{2})}$
\end{proposition}
\begin{proof}

i) From Theorem III.2.1 in [1], we have that the symmetric cone of a Jordan algebra is the set of element $x$ in $V$ for which $L(x)$ is positive definite.\\
Let $x$ be in $\Omega$. For $ y\in V(c,1), \ y \neq 0$, we have:
\begin{eqnarray*}
\langle L(P(c)x)(y), y\rangle& =& \langle P(c)(x)y,y\rangle\\
& =&\langle P(c)(x),y^2\rangle\\
& = &\langle x,P(c)y^2\rangle \\
& =& \langle y, xy\rangle > 0
\end{eqnarray*}
Thus $ P(c)\Omega \subseteq \Omega_c$.\\
Now, let $ w \in \Omega_c$, then $ w + (e-c)$ is an element of
$\Omega$. Since $ P(c)(w + (e-c))= P(c)(w) = w$, we obtain that $
\Omega_c \subseteq P(c)\Omega$.

ii) Let $ x \in V(c, 1)$. It is known (see Faraut-Kor\'{a}nyi,
Prop IV.1.1) that $ V(c, 1). V(c, \frac{1}{2})\subseteq  V(c, \frac{1}{2})$,
hence $2L(x)_{|V(c,\frac{1}{2})}$ is an endomorphism of $V(c,\frac{1}{2})$.\\
 As $c$ is an idempotent of rank  $k$, there exit $ c_{1},c_{2},...,c_{k} $
 orthogonal idempotents and $ (\lambda_1,...,\lambda_k)\in \Bbb{R}^{k}$ such that
 $ c= \displaystyle\sum_{i=1}^k c_i$ and $ x= \displaystyle\sum_{i=1}^k \lambda_i
 c_i$, so that $ \det_{c}x= \displaystyle\prod_{i=1}^k \lambda_i$.
 Similarly, since $e-c$ is an idempotent with rank $r-k$,
 there exit $ c_{k+1},c_{k+2},...,c_{r}$ orthogonal idempotents  such that
  $ e-c= \displaystyle\sum_{i=1}^{r-k} c_{k+i}$.
  The system $ (c_i)_{1\leq i\leq r}$ is a Jordan frame of $V$.
  If for $ 1 \leq i \leq k$, we set $ V_{i,k+1}= \displaystyle\bigoplus_{j=k+1}^r V_{ij}$, then
  $ V(c, \frac{1}{2}) = \displaystyle\bigoplus_{i=1}^k V_{i,k+1}$.
  We can easily  show that  $ 2L(x)_{|V_{i,k+1}}= \lambda_i
  Id_{i,k+1}$,
  where $ Id_{i,k+1}$ is the identity on the space $ V_{i,k+1}$.
  As the dimension of $ V_{i,k+1}$ is equal to $(r-k)d$, we have
  that the determinant of $2L(x)_{|V(c,\frac{1}{2})}$ is equal to
  $ \prod_{i=1}^k \lambda_i^{(r-k)d} = (\det_c x)^{(r-k)d}$.

 iii) If $x$  is invertible in $V(c,1)$, then $ \lambda_1,...,\lambda_k $ are different from zero and
 $ x^{-1}= \displaystyle \sum_{p=1}^k \lambda_p^{-1} c_p$.
Therefore, $ 2L(x)_{|V_{i,k+1}}$ is an automorphism of $V_{i,k+1}$ with inverse
$ \lambda_i^{-1} Id_{i,k+1}$ and it follows that $ 2L(x)_{|V{(c,\frac{1}{2})}}$
is an automorphism of $ V{(c,\frac{1}{2})}$ with inverse $ 2L(x^{-1})_{|V{(c,\frac{1}{2})}}$.\\

iv) We have that $ x^{2}= \displaystyle \sum_{p=1}^k \lambda_p^{2}
c_p$ and for all $ 1 \leq i \leq k$, $ 2L(x)_{|V_{i,k+1}}= \lambda_i
Id_{i,k+1}$. Then
\begin{eqnarray*}
L(x)^{2} _{|V_{i,k+1}}&= & \frac{1}{4}\lambda_i^2 Id_{i,k+1}
  = \frac{1}{2}\frac{\lambda_i^2}{2} Id_{i,k+1}
  =\frac{1}{2} L(x^2)_{|V_{i,k+1}}.
\end{eqnarray*}
Thus, we conclude that  $ L(x^2)_{|V(c,\frac{1}{2})}= \frac{1}{2}
L(x^{2})_{|V(c,\frac{1}{2})}$.
\end{proof}\\

Besides, the results shown above, we will use the facts stated in
the following proposition due to Massam and Neher ([6]).
\begin{proposition}
Let $c$ be an idempotent of $V$, $u_{1}$ in $V(c,1)$, $v_{12}$ in $V(c,\frac{1}{2})$, and $u_{0}$, $z_{0}$ in $V(c,0)$. Then\\

i) $\langle u_{1},P(v_{12})z_{0}\rangle =2 \langle v_{12},L(z_{0})L(u_{1})v_{12}\rangle $\\

ii) $L(z_{0})L(u_{1})=L(u_{1})L(z_{0})$\\

iii) If $u_{1}\in \Omega_{c}$ and $z_{0}\in \Omega_{e-c}$, then
$L(u_{1})L(z_{0})_{|V(c,\frac{1}{2})}$ is a positive definite
endomorphism.
\end{proposition}

Throughout, we suppose that the Jordan frame $(c_{i})_{1\leq i\leq
r}$ is fixed in $V$.
 For 1$\leq k\leq r,$ let $P_{k}$ denote the orthogonal
projection on the Jordan subalgebra

\[
V^{(k)}=V(c_{1}+c_{2}+...+c_{k},1),
\]
$\det^{(k)}$ the determinant in the subalgebra $V^{(k)}$ and, for
$x$ in $V,$ $\Delta _{k}(x)=\det^{(k)}(P_{k}(x)).$ The real number
$\Delta _{k}(x)$ is called the principal minor of order $k$ of $x$
with respect to the frame ($c_{i})_{1\leq i\leq r}.$

\ The generalized power with respect to the Jordan frame
($c_{i})_{1\leq i\leq r}$\ is the polynomial function defined in $x$
of $V$\ by

\[
\Delta _{s}(x)=\Delta _{1}(x)^{s_{1}-s_{2}}\Delta
_{2}(x)^{s_{2}-s_{3}}...\Delta _{r}(x)^{s_{r}}.
\]
Note that $\Delta _{s}(x)=(\det (x))^{p}$ if $s=(p,p,....,p)$ with
$p\in \Bbb{R}$, and if $x=\displaystyle\sum_{i=i}^r \lambda
_{i}c_{i},$ then $\Delta _{s}(x)=\lambda _{1}^{s_{1}}\lambda
_{2}^{s_{2}}...\lambda _{r}^{s_{r}}$. It is also easy to see that
$\Delta _{s+s^{/}}(x)=\Delta _{s}(x).\Delta _{s^{/}}(x)$. In
particular, if $m\in \Bbb{R}$ and
$s+m=(s_{1}+m,s_{2}+m,.......,s_{r}+m),\;$we have $\Delta
_{s+m}(x)=\Delta _{s}(x)\det(x)^{m}.$\\
Now for the fixed Jordan frame $(c_{i})_{1\leq i \leq r}$, and for
$1\leq l \leq r$ we define
\begin{equation}\label{O1}
\sigma_{l}=\displaystyle\sum_{i=1}^lc_i,
 \end{equation}
and we suppose that $V(\sigma_{l},1)$ and $V(e-\sigma_{l},1)$ are
respectively equipped with the Jordan frames $(c_{i})_{1\leq i \leq
l}$ and $(c_{i})_{l+1\leq i \leq r}$. Then we have the following
result which allows the calculation of the general power of some
projections. For the proof we refer the reader to Hassairi and Lajmi ([4]).\\
\begin{theorem}\label{CDet}
Let $1\leq l \leq r-1$, and denote $\theta_{0}$ the orthogonal
projection of an element $\theta$ of the cone $\Omega$ on $ V(e-s_l,
1)$. Then

i)
$\Delta_{l}(\theta^{-1})=\det(\theta^{-1})\det_{e-\sigma_{l}}(\theta_{0}),$

ii) for $l+1\leq k \leq r-1$,
$\frac{\Delta_{k+1}(\theta^{-1})}{\Delta_{k}(\theta^{-1})}=\frac{\Delta_{k+1-l}^{e-\sigma_{l}}(\theta_{0}^{-1})}{\Delta_{k-l}^{e-\sigma_{l}}(\theta_{0}^{-1})}$,
and
$\frac{\Delta_{l+1}(\theta^{-1})}{\Delta_{l}(\theta^{-1})}=\Delta_{1}^{e-\sigma_{l}}(\theta_{0}^{-1})$.
\end{theorem}

We now introduce the set $\Xi$ of elements $s=(s_{1},...,s_{r})$ in $\Bbb{R}^{r}$ defined in the following way:\\
For a given real number $u\geq0$, we set
 $$\varepsilon(u)=0 \ \ \ \textrm{if} \ \ \ u=0$$
$$\varepsilon  (u)=1 \ \ \ \textrm{if} \ \ \ u>0$$
Given $u=(u_{1},...,u_{r})\in \mathbb{R}_{+}^{r}$, we define
\begin{equation}\label{S3}
s_{1}=u_{1}\ \ \textrm{and} \ \
s_{i}=u_{i}+\frac{d}{2}(\varepsilon(u_{1})+...+\varepsilon(u_{i-1}))\
\ \textrm{for} \ \ \ 2\leq i \leq r.
\end{equation}
Note that the set $\Xi$ contains
$\prod_{i=1}^{r}](i-1)\frac{d}{2},+\infty[$, and that
$$\Lambda=\{p\in \Bbb{R} \ \ \textrm{such that} \ \ (p,...,p)\in \Xi \}=\left \{\frac{d}{2},...,\frac{d}{2}(r-1)\right \}\cup ](r-1)\frac{d}{2},+\infty[$$
The definition of the Riesz measure is based on the following
theorem due to Gindikin ([2]), for a proof we refer the reader to Faraut and Kor\'{a}nyi (1994). The Laplace transform of a positive measure $\mu$ on $V$
is defined by
$$ L_\mu(\theta)=\displaystyle\int_{V} \exp(\langle\theta, x\rangle)\mu(dx).$$

\begin{theorem}\label{ST5}There exists a positive measure $R_{s}$ on
$V$ with Laplace transform defined on -$\Omega$ by
$L_{R_{s}}(\theta)=\Delta_{s}(-\theta^{-1})$ if and only if $s$ is
in the set $\Xi$.
\end{theorem}
Hassairi and Lajmi ([3]) have  called the measure $R_{s}$ Riesz measure
and they have used it to introduce a class of probability
distributions which is an important extension of the famous Wishart
ones. When $s=(s_{1},s_{2},.....,s_{r})$ is in
$\prod_{i=1}^{r}](i-1)\frac{d}{2},+\infty[$, the measure $R_{s}$ has
an explicit expression. In fact, if for $s$ such that for all $i$,
$s_{i}>(i-1)\frac{d}{2}$, we consider the measure
\[
R_{s}=\frac{1}{\Gamma _{\Omega }(s)}\Delta _{s\text{-}\frac{n}{r}}(x)\mathbf{%
1}_{\Omega }(x)dx,
\]
where $\Gamma _{\Omega }(s)=(2\pi )^{\frac{n-r}{2}}{
{\prod _{j=1}^r}}\Gamma (s_{j}-(j-1)\frac{d}{2})$, then it is proved in
Faraut Kor\'{a}nyi (1994), Theo. VII.1.2, that the Laplace
transform of $R_{s}$ is equal to $\Delta _{s}(-\theta ^{-1})\;$for
$\theta \in -\Omega ,$ that is for all $\theta\in -\Omega$,
$$\displaystyle\frac{1}{\Gamma _{\Omega }(s)}\int \exp(\langle\theta, x\rangle)\Delta _{s\text{-}\frac{n}{r}}(x)\mathbf{%
1}_{\Omega }(x)(dx)=\Delta_{s}(-\theta^{-1}).$$

\section{Description of the Riesz measures}

In this section, we give a complete description of the Riesz mesure
$R_s$ inclosing the ones corresponding to $s$ in
$\Xi\backslash\prod_{i=1}^{r}](i-1)\frac{d}{2},+\infty[$ which are
concentrated on the boundary $\partial\Omega$ of the symmetric cone
$\Omega$. In order to do so, we need to recall some facts on the
boundary structure of the cone $\Omega$. More precisely, we have the
following useful decomposition of the closed cone
$\overline{\Omega}$ into orbits under the action of the group $G$,
connected component of the identity in $ G(\Omega)$, which appears
in Lasalle (1987) and in Faraut and Kor\'{a}nyi (1994). Recall that
for the fixed Jordan frame $(c_{i})_{1\leq i \leq r}$ and $1\leq l
\leq r$, $\sigma_{l}=\displaystyle\sum_{i=1}^lc_i.$

\begin{proposition}

i) Let $x$ be in $\overline{\Omega}$. Then $x$ is of rank $l$ if and
only if $x\in G{\sigma_{l}}$

ii) We have that
$\overline{\Omega}=\displaystyle\bigcup_{l=1}^rG{\sigma_{l}}$.
\\More precisely, $\Omega=G{\sigma_{r}}=G{e}$ and
$\partial\Omega=\displaystyle\bigcup_{l=1}^{r-1}G{\sigma_{l}}$

iii) Denote for  $1\leq l \leq r-1$
\begin{center}
    $J_{l}=\{x\in G{\sigma_{l}}; \ \Delta_{l}(x)\neq 0\}$.
\end{center}
then $J_l $ is an open subset dense in $G{\sigma_{l}}$.

iv)Suppose that $x=x_{1}+x_{12}+x_{0}$ is the Peirce decomposition
of $x$ with respect to $\sigma_{l}$,  then the map
\begin{center}
    $\Omega_{\sigma_{l}} \times V(\sigma_{l},\frac{1}{2})\rightarrow J_{l}$ \
    ; \  $(x_{1},x_{12})\longmapsto
        x_{1}+x_{12}+2(e-\sigma_{l})[x_{12}(x_{1}^{-1}x_{12})]$
\end{center}
is a bijection.
\end{proposition}
As a corollary of the last point, we have that an element $x$ of
$J_{l}$ can be written in a unique way as
$x=x_{1}+x_{12}+(e-\sigma_{l})v^{2}$, where
$v=\frac{1}{2}\sqrt{x_{1}^{-1}}x_{12}$.\\

We now give the description of the Riesz measures $R_{s}$ when $s$
has a particular form, we then give the general case.
\begin{theorem}
Let $l$ be in $\{1,...,r\}$,
$\sigma_{l}=\displaystyle\sum_{i=1}^lc_i$, and $u=(u_{1},...,u_{l})$
in $\mathbb{R}^{l}$ such that $u_{i}>(i-1)\frac{d}{2}$, for $1\leq i
\leq l$. Consider the measure
$$\gamma_{l}(dx_{1},dv)=\frac{\Delta_{u}^{\sigma_{l}}(x_{1})(\det_{\sigma_{l}}(x_{1}))^{-1-(l-1)\frac{d}{2}}}{(2\pi)^{l(r-l)\frac{d}{2}}\Gamma_{\Omega_{\sigma_{l}}}(u)}{\mathbf{1}}_{\Omega_{\sigma_{l}} \times V(\sigma_{l},\frac{1}{2})}(x_{1},v)dx_{1}dv$$
and the map
\begin{center}
    $\alpha : \Omega_{\sigma_{l}} \times V(\sigma_{l},\frac{1}{2})\rightarrow V$ \
    ; \  $(x_{1},v) \longmapsto
    x_{1}+2v\sqrt{x_{1}}+(e-\sigma_{l})v^{2}$.
\end{center}
Then the Laplace transform of the image $\mu_{l}=\alpha \gamma_{l}$
of $\gamma_{l}$ by $\alpha$ is defined on $-\Omega$ and is given by
$$L_{\mu_{l}}(\theta)=\Delta_{s}(-\theta^{-1}),$$
where
$s=(u_{1},...,u_{l},\frac{dl}{2},...,\frac{dl}{2})\in\mathbb{R}^{r}$.
\end{theorem}
\begin{proof}
Let $\theta$ be in $-\Omega$ and let
$\theta=\theta_{1}+\theta_{12}+\theta_{0}$ be its Peirce
decomposition  with respect to $\sigma_{l}$. Then according to
Proposition 2.1, i), we have that $\theta_{1}=P(\sigma_{l})(\theta)$
is in $-\Omega_{\sigma_{l}}$ and
$\theta_{0}=P(e-\sigma_{l})(\theta)$ is in $-\Omega_{e-\sigma_{l}}$.
Let us calculate the Laplace transform of $\mu_{l}$ in $\theta$.
\begin{eqnarray*}
 L_{\mu_{l}}(\theta)&=& \int_{\Omega_{\sigma_{l}}\times V(\sigma_{l},\frac{1}{2})} \exp(\langle \theta,\alpha(x_{1},v)\rangle \gamma_{l}(dx_{1},dv)\\
 &=& \int_{\Omega_{\sigma_{l}}\times V(\sigma_{l},\frac{1}{2})}\exp(\langle \theta_{1},x_{1}\rangle +\langle \theta_{12},2v\sqrt{x_{1}} \rangle + \langle \theta_{0},v^{2}\rangle )\\& & \frac{\Delta_{u}^{\sigma_{l}}(x_{1})\det_{\sigma_{l}}(x_{1})^{-1-(l-1)\frac{d}{2}}}{(2\pi)^{l(r-l)\frac{d}{2}}\Gamma_{\Omega_{\sigma_{l}}}(u)}
 dx_{1}dv.
 \end{eqnarray*}
 This may be written as
\begin{equation}\label{S3}
L_{\mu_{l}}(\theta) = \int_{\Omega_{\sigma_{l}}}I(x_{1})\exp(\langle
\theta_{1},x_{1}\rangle
)\Delta_{u}^{\sigma_{l}}(x_{1})\textrm{det}_{\sigma_{l}}(x_{1})^{-1-(l-1)\frac{d}{2}}\frac{dx_{1}}{\Gamma_{\Omega_{\sigma_{l}}}(u)},
\end{equation}
 where $$I(x_{1})=\frac{1}{(2\pi)^{l(r-l)\frac{d}{2}}}\int_{ V(\sigma_{l},\frac{1}{2})}\exp(\langle \theta_{12},2v\sqrt{x_{1}} \rangle +\langle \theta_{0},v^{2}\rangle
 )dv.$$
 According to Prop 2.2, iii) and Prop 2.1, iii)
 , we have that $2L(-\theta_{0})_{|
 V(\sigma_{l},\frac{1}{2})}=L(4\sigma_{l})L(-\theta_{0})_{|V(\sigma_{l},\frac{1}{2})}$
 is an automorphism of $V(\sigma_{l},\frac{1}{2})$ whose the inverse is equal to $2L(-\theta_{0}^{-1})_{|
 V(\sigma_{l},\frac{1}{2})}$. Thus, one can write
 $$I(x_{1})=\frac{1}{(2\pi)^{l(r-l)\frac{d}{2}}}\int_{ V(\sigma_{l},\frac{1}{2})}\exp(\langle 2\sqrt{x_{1}}\theta_{12},v \rangle -\frac{1}{2}\langle 2L(-\theta_{0})v,v\rangle
 )dv.$$
 Using Lemma VII.2.5 Faraut and Kor\'{a}nyi (1994), then again Proposition 2.1, we get
 \begin{eqnarray*}
 I(x_{1})&=& \left(\det 2L(-\theta_{0}^{-1})_{|
 V(\sigma_{l},\frac{1}{2})}
 \right)^{\frac{1}{2}}\exp(\frac{1}{2}\langle 2\sqrt{x_{1}}\theta_{12},2L(-\theta_{0}^{-1})2\sqrt{x_{1}}\theta_{12}\rangle)
 \\&=& \det_{{e-\sigma_{l}}}(-\theta_{0}^{-1})^{l\frac{d}{2}}\exp(\frac{1}{2}\langle 2\sqrt{x_{1}}\theta_{12},4L(-\theta_{0}^{-1})L(\sqrt{x_{1}})\theta_{12}\rangle).
 \end{eqnarray*}
As $L(\sqrt{x_{1}})$ is symmetric, we can write
$$I(x_{1})= \det_{e-\sigma_{l}}(-\theta_{0}^{-1})^{l\frac{d}{2}}\exp(\langle 4\theta_{12},L(\sqrt{x_{1}})^{2}L(-\theta_{0}^{-1})\theta_{12}\rangle ).$$
Proposition 2.1 implies that
$$I(x_{1})= \det_{e-\sigma_{l}}(-\theta_{0}^{-1})^{l\frac{d}{2}}\exp(2\langle \theta_{12},L(x_{1})L(-\theta_{0}^{-1})\theta_{12}\rangle ).$$
Finally, from Proposition 2.2, we deduce that
$$I(x_{1})=\det_{e-\sigma_{l}}(-\theta_{0}^{-1})^{l\frac{d}{2}}\exp(\langle x_{1},P(\theta_{12})(-\theta_{0}^{-1})\rangle ).$$
Now inserting this in $(3)$, we obtain
\begin{eqnarray*}
 L_{\mu_{l}}(\theta)&=&\frac{\det _{e-\sigma_{l}}(-\theta_{0})^{-l\frac{d}{2}}}{\Gamma_{\Omega_{\sigma_{l}}}(u)}
\int_{\Omega_{\sigma_{l}}}\exp(\langle
x_{1},\theta_{1}-P(\theta_{12})(\theta_{0}^{-1})\rangle
)\Delta_{u}^{\sigma_{l}}(x_{1})\\& & \det
_{\sigma_{l}}(x_{1})^{-1-(l-1)\frac{d}{2}}dx_{1}
 \\&=& \det _{e-\sigma_{l}}(-\theta_{0})^{-l\frac{d}{2}}\Delta_{u}^{\sigma_{l}}\left(-(\theta_{1}-P(\theta_{12})(\theta_{0}^{-1}))^{-1}\right).
 \end{eqnarray*}

Since $(\theta^{-1})_{l}=P(\sigma_{l})(\theta^{-1})= (\theta_1 -
P(\theta_{12})(\theta_{0}^{-1}))^{-1}$ and according to Theorem 2.3,
we can write
\begin{eqnarray*}
 L_{\mu_{l}}(\theta)&=& \det _{e-\sigma_{l}}(-\theta_{0})^{-l\frac{d}{2}}\Delta_{u}^{\sigma_{l}}(-(\theta^{-1})_{l})
 \\&=&
 \left(\frac{\Delta_{l}(-\theta^{-1})}{\det(-\theta^{-1})}\right)^{-l\frac{d}{2}}\Delta_{u}^{\sigma_{l}}(-(\theta^{-1})_{l}),
\end{eqnarray*}
Therefore,
\begin{eqnarray*}
L_{\mu_{l}}(\theta)&=&
\Delta_{1}(-\theta^{-1})^{u_{1}-u_{2}}...\Delta_{l-1}(-\theta^{-1})^{u_{l-1}-u_{l}}\Delta_{l}(-\theta^{-1})^{u_{l}-l\frac{d}{2}}\det(-\theta^{-1})^{l\frac{d}{2}}
\\&=&\Delta_{s}(-\theta^{-1}),
\end{eqnarray*}
where $s=(u_{1},...,u_{l},l\frac{d}{2},...,l\frac{d}{2})$ in
$\mathbb{R}^{r}$
\end{proof}
\begin{corollary} For $ 1 \leq l \leq r-1$, the measure $ \mu_{l}$ is concentrated  on the boundary $ \partial\Omega$ of the symmetric cone $\Omega$.
\end{corollary}
\begin{proof}
 In fact, $\mu_{l}$ is concentrated
on on the set $ J_{l}=\{x\in G{\sigma_{l}}; \ \Delta_{l}(x)\neq 0\}$
which is dense in $G_{\sigma_{l}}$ (prop 3.1)).
\end{proof}
\begin{theorem}
Let $l$ be in $\{1,...,r-1\}$, and suppose that for
$u=(u_{1},...,u_{r-l})\in\mathbb{R}_{+}^{r-l}$, there exists a
measure $\mu_{u}$ on $V(e-\sigma_{l},1)$ such that the Laplace
transform is defined on $-\Omega_{e-\sigma_{l}}$ and is equal to
$\Delta_{u}^{e-\sigma_{l}}(-\theta_{0}^{-1})$. Then the Laplace
transform of the measure $\mu$ image of $\mu_{u}$ by the injection
of $V(e-\sigma_{l},1)$ into $V$ is defined on $-\Omega$ and
$L_{\mu}(\theta)=\Delta_{s}(-\theta^{-1})$, where
$s=(0,...0,u_{1},...,u_{r-l})\in\mathbb{R}_+ ^{r}$.
\end{theorem}
\begin{proof}
Let $x=x_{1}+x_{12}+x_{0}$ and
$\theta=\theta_{1}+\theta_{12}+\theta_{0}$ be respectively the
Peirce decomposition with respect to $\sigma_{l}$ of an element $x$
of $V$ and an element $\theta$ of $-\Omega$. Then
\begin{eqnarray*}
L_{\mu}(\theta)&=&\int_{V(e-\sigma_{l},1)}\exp(\langle
\theta,x_{0}\rangle)
\mu_{u}(dx_{0})\\&=&\int_{V(e-\sigma_{l},1)}\exp(\langle
\theta_{0},x_{0}\rangle) \mu_{u}(dx_{0})
\\&=&\Delta_{u}^{e-\sigma_{l}}(-\theta^{-1}_{0})\\&=&\Delta_{1}^{e-\sigma_{l}}(-\theta^{-1}_{0})^{u_{1}}\left(\frac{\Delta_{2}^{e-\sigma_{l}}(-\theta^{-1}_{0})}{\Delta_{1}^{e-\sigma_{l}}(-\theta^{-1}_{0})}\right)^{u_{2}}...\left(\frac{\Delta_{r-l}^{e-\sigma_{l}}(-\theta^{-1}_{0})}{\Delta_{r-l-1}^{e-\sigma_{l}}(-\theta^{-1}_{0})}\right)^{u_{r-l}}
\end{eqnarray*}
This according to Theorem 2.3 leads to
\begin{eqnarray*}
L_{\mu}(\theta)&=&\left(\frac{\Delta_{l+1}(-\theta^{-1})}{\Delta_{l}(-\theta^{-1})}\right)^{u_{1}}\left(\frac{\Delta_{l+2}(-\theta^{-1})}{\Delta_{l+1}(-\theta^{-1})}\right)^{u_{2}}...\left(\frac{\Delta_{r}(-\theta^{-1})}{\Delta_{r-1}(-\theta^{-1})}\right)^{u_{r-l}}
\\&=&\Delta_{l}(-\theta^{-1})^{-u_{1}}\Delta_{l+1}(-\theta^{-1})^{u_{1}-u_{2}}...\Delta_{r}(-\theta^{-1})^{u_{r-l}}
\\&=&\Delta_{s}(-\theta^{-1}),
\end{eqnarray*}
where $s=(0,...,0,u_{1},...,u_{r-l})\in \mathbb{R}_{+}^{r}$.
\end{proof}

We come now to the construction of the Riesz measure $R_{s}$ for any
$s=(s_{1},...,s_{r})$ in the set $\Xi$. From the definition of
$\Xi$, there exists $u=(u_{1},...,u_{r})\in \mathbb{R}_{+}^{r}$ such
that
\begin{equation}\label{S4}
s_{1}=u_{1}\ \ \textrm{and} \ \
s_{i}=u_{i}+\frac{d}{2}(\varepsilon(u_{1})+...+\varepsilon(u_{i-1})).
\end{equation}
We will use $(u_{1},...,u_{r})$, to construct a partition $(A_{i})$
of the set $\{1,...,r\}$ such that, for all $i$, we have either
$u_{j}=0, \forall j \in A_{i}$ or $u_{j}>0, \forall j \in A_{i}$.
Such a partition is important in the description of the measure
$R_{s}$.

Consider the sequences of integers $i_{1} ,...,i_{k}$ and $j_{1}
,...,j_{k}$  built as follows:\\
$i_{1}=\inf\{p\geq 0   \ \ \textrm{such that} \ \  u_{p+1}\neq 0\}$,\\
$j_{l}=\inf \{p\geq 0  \ \textrm{such that} \ \
 u_{{i_l}+p+1}=0 \ \}  ,\,\,\, 1 \leq l \leq k $, \\
$i_{l}=\inf\{p\geq i_{l-1}+ j_{l-1}  \ \ \textrm{such that} \ \  u_{p+1}\neq 0\} , \,\,\,  2 \leq l \leq k $,\\

In this way, we get a partition of $u=(u_{1},...,u_{r})$ in the
form:
$$u=(\underset{i_{1}\ \ terms}{\underbrace{0,...,0}},\underset{j_{1}\ \ terms}{\underbrace{u_{i_{1}+1},...,u_{i_1 +j_{1}}}},..., {\underbrace{0,...,0}},\underset{j_{l}\ \ terms}{\underbrace{u_{i_{l}+1},...,u_{i_l + j_{l}}}},..., {\underbrace{0,...,0}},\underset{j_{k}\ \ terms}{\underbrace{u_{i_{k}+1},...,u_{i_k + j_{k}}}},...).$$
This partition of $u$ leads to the following partition of the set
$\{1,...,r\}$ defined by
$$ I'_{0}=\left \{ \begin{array}{l}
\emptyset \ \ \ \ \ \ \ \ \ \ \ \  if \ \ \ i_{1}=0 \\
\{1,...,i_{1}\}\ \ if \ \ \ i_{1}\neq0
\end{array} \right.$$
$$ \ \ \ \ \ \ \ \ \ \ \ I'_{l}=\{i_{l}+j_{l}+1,...,i_{l+1}\}\ \ if \ \ \ 1 \leq l\leq k-1 .$$
$$\ \ \ \ \ \ \ \ \ \ \ \ \ \ \ \ \ \ I'_{k}=\left \{ \begin{array}{l}
\emptyset \ \ \ \ \ if \ \ \ i_{k}+j_{k}=r \\
\{i_{k}+j_{k}+1,...,r\}\ \ if \ \ \ i_{k}+j_{k}<r\\
\end{array} \right.$$
and
$$\ \ \ \ \ \ \ \ \ \ \ \  I_{l}=\{i_{l}+1,...,i_{l}+j_{l}\}\ \ if \ \ \ 1 \leq l \leq k.$$
Thus we have that
$$\underset{1 \leq p \leq k}{\bigcup} I_{p}=\{i\ \ ; \ \ u_{i}\neq0\} \
\ \textrm{and}\ \ \underset{0 \leq p \leq k}{\bigcup} I'_{p}=\{i\ \
; \ \ u_{i}=0\}.$$ In conclusion, for an element $s$ in $\Xi$, we
associate $u=(u_{1},...,u_{r})$, $k$ in $\{1,...,r\}$, and the
partition of the set $\{1,...,r\}$ defined above. We also define for
$1 \leq l \leq k$,
$$u^{(l)}=\left(u_{i_{l}+1},u_{i_{l}+2}+\frac{d}{2},...,u_{i_{l}+j_{l}}+\frac{d}{2}(j_{l}-1)\right),$$
which is in $\mathbb{R}^{j_{l}}$ and the element of $\mathbb{R}^{r}$
\begin{equation}\label{S5}
s^{(l)}=\left(\underset{i_{l}\ \
terms}{\underbrace{0,...,0}},u^{(l)},\frac{d}{2}j_{l},...,\frac{d}{2}j_{l}\right),
\end{equation}
which can be written as
$$s^{(l)}=(\alpha^{(l)}_{1},...,\alpha^{(l)}_{r})$$
with
$$\left \{ \begin{array}{l}
\alpha^{(l)}_{p}=0 \ \ \ \ \ \ \ \ \ \ \ \ \ \ \ \ \ \ \ \ \ \ \ \ \ if \ \ \ 1 \leq p \leq i_{l} \\
\alpha^{(l)}_{i_{l}+p}=u_{i_{l}+p}+\frac{d}{2}(p-1) \ \ \ \ if \ \
\ 1 \leq p \leq j_{l}\\
\alpha^{(l)}_{p}=\frac{d}{2}j_l  \ \ \ \ \ \ \ \ \ \ \ \ \ \ \ \ \ \
\ \ \ \ \ \ if \ \ \ i_{l}+j_{l}+1 \leq p \leq r.
\end{array} \right.$$
The last term disappears if $i_{l}+j_{l}=r$.
\begin{proposition}
With the previous notations, for any $s$ in $\Xi$, we have
$$s=\underset{1 \leq l \leq k}{\sum}s^{(l)}.$$
\end{proposition}
\begin{proof}
Recall that the corresponding vector $u=(u_{1},...,u_{r})$ to a
given $s$ in $\Xi$ is such that
$$s_{1}=u_{1}\ \ \textrm{and} \ \
s_{i}=u_{i}+\frac{d}{2}(\varepsilon(u_{1})+...+\varepsilon(u_{i-1}))$$
 Given $m$ in
$\{1,...,r\}$, we distinguish between four cases according to its
position in the elements $I'_{0}$, $I'_{l}$, $I'_{k}$ and $I_{l}$ of
the partition of $\{1,...,r\}$.

If $m\in I'_{0}$, then $s_{m}=0$ and $\alpha^{(l)}_{m}=0$, for $1
\leq l \leq k$, since $ m\leq i_1 \leq i_l
 $ so that we have  $$s_{m}=\underset{1 \leq l \leq
k}{\sum}\alpha^{(l)}_{m}$$.

If $m\in I_{l}$ with $ 1\leq l \leq k$, then $i_{l}+1 \leq m \leq i_{l}+j_{l}$, . It follows
that
$$\left \{ \begin{array}{l}
\alpha^{(p)}_{m}=\frac{d}{2}j_{p} \ \ \ \ \ \ \ if \ \ \ 1 \leq p \leq l-1, \ \ \ \ since \ \ i_{p}+j_{p} \leq i_{l-1}+j_{l-1} < i_{l}<m \\
\alpha^{(p)}_{m}=u_{m}+\frac{d}{2}(m-i_{l}-1) \ \ \ \ \ \ \ if \ \ \ p=l, \ \ \ \ since \ \ 1 \leq m-i_{l} \leq j_{l}\\
\alpha^{(p)}_{m}=0 \ \ \ \ \ \ \ if \ \ \ l+1\leq p \leq k, \ \ \ \
since \ \ m \leq i_{l}+j_{l} < i_{l+1}<i_{p}.
\end{array} \right.$$
Therefore $$\underset{1 \leq p \leq
k}{\sum}\alpha^{(p)}_{m}=u_m +\frac{d}{2}(m-i_{l}-1 +j_{1}+...+j_{l-1})=s_{m}$$.

If $m\in I'_{l}$, with $1 \leq l \leq k-1$, then $ i_{l}+j_{l}+1 \leq
m \leq i_{l+1}$. Il follows that
$$\left \{ \begin{array}{l}
\alpha^{(p)}_{m}=\frac{d}{2}j_{p} \ \ \ \ \ \ \ if \ \ \ 1 \leq p \leq l, \ \ \ \ since \ \ i_{1}+j_{1}<... < i_{l}+j_{l} < m \\
\alpha^{(p)}_{m}=0 \ \ \ \ \ \ \ if \ \ \ l+1\leq p \leq k, \ \ \ \
since \ \ m \leq i_{l+1}<i_{l+2}<... <i_{k}.
\end{array} \right.$$
As $u_{m}=0$, we obtain
$$\underset{1 \leq p \leq k}{\sum}\alpha^{(p)}_{m}=\frac{d}{2}(j_{1}+...+j_{l})=s_{m}$$.

If $m\in I'_{k}$, then $ i_{k}+j_{k}+1 \leq m \leq r$. Since
$i_{1}+j_{1}<... < i_{k}+j_{k} < m$, it follows that
$$\alpha^{(p)}_{m}=\frac{d}{2}j_{p} \  \ 1 \leq p
\leq k.$$ Thus
$$\underset{1 \leq p \leq
k}{\sum}\alpha^{(p)}_{m}=\frac{d}{2}(j_{1}+...+j_{k})=s_{m}$$.
\end{proof}

To continue our description of the Riesz Mesures, we require some
further notations. For $s$ in $\Xi$, and $1\leq l \leq k$, where $k$
is the integer corresponding to $s$ defined above, we set

$$\overline{c}_{i_{l}}=c_{i_{l}+1}+...+c_{r}$$
$$\ \ \ \overline{c}_{i_{l},j_{l}}=c_{i_{l}+1}+...+c_{i_{l}+j_{l}}$$
$\overline{c}_{i_{l},j_{l}}$ is an idempotent of rank $j_{l}$ in
$V(\overline{c}_{i_{l}},1)$. \\Let
$\widehat{V}(\overline{c}_{i_{l},j_{l}},1)$
 and $\widehat{V}(\overline{c}_{i_{l},j_{l}},\frac{1}{2})$ be the subspaces of $V(\overline{c}_{i_{l}},1)$ corresponding to
 the eigenvalues $1$ and $\frac{1}{2}$, and let
 $\widehat{\Omega}_{\overline{c}_{i_{l},j_{l}}} $ be the symmetric cone associated to
 $\widehat{V}(\overline{c}_{i_{l},j_{l}},1)$.

Consider the map
\begin{center}
    $\alpha : \widehat{\Omega}_{\overline{c}_{i_{l},j_{l}}} \times \widehat{V}(\overline{c}_{i_{l},j_{l}}, \frac{1}{2})\rightarrow V( \overline{c}_{i_{l}},1) $ \
    ; \  $(x,v) \longmapsto
    x+2v\sqrt{x}+(\overline{c}_{i_{l}}-\overline{c}_{i_{l},j_{l}})v^{2}$,
\end{center}
and let $i$ be the canonical injection of
$V(\overline{c}_{i_{l}},1)$ into $V$. \\We now define the measure
$$\gamma_{u^{(l)}}(dx,dv)=\frac{\Delta_{u^{(l)}}^{\overline{c}_{i_{l},j_{l}}}(x)(\det_{\overline{c}_{i_{l},j_{l}}}x)^{-1-(j_{l}-1)\frac{d}{2}}}{(2\pi)^{j_{l}(r-i_{l}-j_{l})\frac{d}{2}}{\Gamma_{\widehat{\Omega}_{\overline{c}_{i_{l},j_{l}}}}}(u^{(l)})}\mathbf{1}_{\widehat{\Omega}_{\overline{c}_{i_{l},j_{l}}} \times\widehat{V}(\overline{c}_{i_{l},j_{l}},\frac{1}{2})}(x,v)dxdv ,$$
and we denote $\mu_{u^{(l)}}$ the image of $\gamma_{u^{(l)}}$ by the
map $i \circ \alpha$. \\ We are now ready to state and prove our
main result.
\begin{theorem}
For all $s$ in $\Xi$, we have
$$R_{s}=\mu_{u^{(1)}}\star...\star\mu_{u^{(k)}},$$
where $\star$ is the convolution product.
\end{theorem}
\begin{proof}
We need to show that the Laplace transform of
$\mu_{u^{(1)}}\star...\star\mu_{u^{(k)}}$ defined in an element
 $\theta$ of $-\Omega$ is equal to $\Delta_s(-\theta^{-1})$. \\
 For $1\leq l\leq k$, let $\theta=\theta_{1}+\theta_{12}+\theta_{0}$ be the Peirce
decomposition of $\theta$ with respect to $\overline{c}_{i_{l}}$. If
we denote $\mu'_{u^{(l)}}$ the image of $\gamma_{u^{(l)}}$ by the
map $\alpha$, then according to Theorem $(3.2)$, we have that
$$L_{\mu'_{u^{(l)}}}(-\theta_{0})=\Delta_{s'^{(l)}}^{\overline{c}_{i_{l}}}(-\theta_{0}^{-1}),$$
where $s'^{(l)}=(u^{(l)},\frac{d}{2}j_{l},...,\frac{d}{2}j_{l})\in
\mathbb{R}^{r-i_{l}}$. On the other hand, as $\mu_{u^{(l)}}$ is the
image of $\mu'_{u^{(l)}}$ by the canonical injection of
$V(\overline{c}_{i_{l}},1)$ into $V$, Theorem $(3.4)$ implies that
$$L_{\mu_{u^{(l)}}}(\theta_{0})=\Delta_{s^{(l)}}(-\theta^{-1}),$$
where $s^{(l)}=(0,...0,s'^{(l)})\in \mathbb{R}^{r}$. \\Therefore the
Laplace transform of $\mu_{u^{(1)}}\star...\star\mu_{u^{(k)}}$ in
$\theta \in -\Omega$ is
\begin{eqnarray*}L_{\mu_{u^{(1)}}\star...\star\mu_{u^{(k
)}}}(\theta)&=&\underset{1 \leq l \leq
k}{\Pi}L_{\mu_{u^{(l)}}}(\theta)\\
&=&\underset{1 \leq l \leq k}{\Pi}\Delta_{s^{(l)}}(-\theta^{-1})
\\
&=&\Delta_{\underset{1 \leq l \leq k}{\sum}s^{(l)}}(-\theta^{-1})\\
&=&\Delta_{s}(-\theta^{-1}),
\end{eqnarray*}
which is the desired result
\end{proof}
\begin{corollary}
a) The measure $\mu_{u^{(l)}}$ is supported by the set
$$J'_{u^{(l)}}=\{x\in V(\overline{c}_{i_{l}},1)\ \ \textrm {such that}\ \  x\in\overline{\Omega} \ \ \textrm{ and}\ \ rank x=j_{l} \}$$
\ \ \ \ \ \ \ \ \ \ \ \ \ \ \ \ \ \ \ \ b) The measure $R_{s}$ is
supported by the set
$$J'_{u^{(1)}}+...J'_{u^{(k)}}\subseteq V(\overline{c}_{1},1)\cap\overline{\Omega}$$
\end{corollary}
\begin{proof}

a) Follows from Corollary 1.1.

b) It suffises to observe that $V(\overline{c}_{1},1)\supset
V(\overline{c}_{2},1)\supset ...\supset V(\overline{c}_{k},1)$
\end{proof}
\begin{remark}

a) When $s$ is in $\Xi$ such that $s_{i}>(i-1)\frac{d}{2}
, \ \ 1\leq
i \leq r$, then the integer $k$ corresponding to $s$ is equal to 1.
In this case $R_{s}=\mu_{u^{(1)}}$, it is concentrated on $\Omega$.

\ \ \ \ \ \ \ \ \ \ \ \ b) When $s$ is in
$\Xi\backslash\prod_{i=1}^{r}](i-1)\frac{d}{2},+\infty$, , then the
integer $k$ corresponding to $s$ is strictly greater than 1 and
$j_{1}+...+j_{k}<r$. The measure $R_s$ is in this case concentrated
on $J'_{u^{(1)}}+...J'_{u^{(k)}}$ whose the element are of rank less
than or equal to $j_{1}+...+j_{k}$. As $j_{1}+...+j_{k}<r$, $R_s$ is
supported by the boundary $\partial\Omega$ of the symmetric cone
$\Omega$.
\end{remark}

\noindent \textbf{References}\\
$[1]$ J. Faraut, A. Kor\'{a}nyi. (1994). Analysis on symmetric cones, Oxford Univ,
Press.\\
$[2]$ S.G. Gindikin. (1964). Analysis on homogeneous domains. Russian Math.
Surveys. 29,) 1-89.\\
$[3]$ A. Hassairi, S.Lajmi. (2001). Riesz exponential families on symmetric cones.
J. Theoret. Probab. Vol 14, 927-948.\\
$[4]$ A. Hassairi, S.Lajmi. (2004). Classification of Riesz exponential families on a
symmetric cone by invariance properties. J. Theoret. Probab. Vol 17, No.3.\\
$[5]$ M. Lassalle. (1987). Algèbre de Jordan et ensemble de Wallah. Invent. Math.
 89, 375-393.\\
$[6]$ H. Massam, E. Neher. (1997). On transformation and determinants of Wishart
variables on symmetric cones. J. Theoret. Probab. Vol 10, 867-902.

\end{document}